\documentclass[a4paper]{article}
\usepackage{amsmath}
\usepackage{amsthm}
\usepackage{amssymb}
\usepackage{graphicx}
\usepackage{authblk}
\usepackage{hyperref}
\usepackage{multirow}
\usepackage{rotating}
\usepackage{adjustbox}

\newcommand{\num}{L}

\newtheorem{definition}{Definition}
\newtheorem{theorem}{Theorem}

\newtheorem{lemma}{Lemma}

\newtheorem{rem}{Remark}
\newtheorem{examp}{Example}

\usepackage[symbol]{footmisc}

\title{The box dimension of degenerate spiral trajectories of a class of ordinary differential equations}

\author{Renato Huzak, Domagoj Vlah, Darko \v Zubrini\'c and Vesna \v Zupanovi\'c}

\begin{document}

\begin{titlepage}

\end{titlepage}

\maketitle

\begin{abstract}
In this paper we initiate the study of the box dimension of degenerate spiral trajectories of a class of ordinary differential equations. A class of singularities of focus type with two zero eigenvalues (nilpotent or more degenerate) has been studied. We find the box dimension of a polynomial degenerate focus of type $(n,n)$ by exploiting the well-known fractal results for $\alpha$-power spirals. In the general $(m,n)$ case, we formulate a conjecture about the box dimension of a degenerate focus.
Further, we reduce the fractal analysis of planar nilpotent contact points to the study of the box dimension of a slow-fast spiral generated by their ``entry-exit" function. There exists a bijective correspondence between the box dimension of the slow-fast spiral and the codimension of contact points. We also construct a three-dimensional vector field that contains a degenerate spiral, called an elliptical power spiral, as a trajectory. 
\end{abstract}

\font\csc=cmcsc10

\def\esssup{\mathop{\rm ess\,sup}}
\def\essinf{\mathop{\rm ess\,inf}}
\def\wo#1#2#3{W^{#1,#2}_0(#3)}
\def\w#1#2#3{W^{#1,#2}(#3)}
\def\wloc#1#2#3{W_{\scriptstyle loc}^{#1,#2}(#3)}
\def\osc{\mathop{\rm osc}}
\def\var{\mathop{\rm Var}}
\def\supp{\mathop{\rm supp}}
\def\Cap{{\rm Cap}}
\def\norma#1#2{\|#1\|_{#2}}

\def\C{\Gamma}

\let\text=\mbox

\catcode`\@=11
\let\ced=\c
\def\a{\alpha}
\def\b{\beta}
\def\c{\gamma}
\def\d{\delta}
\def\g{\lambda}
\def\o{\omega}
\def\q{\quad}
\def\n{\nabla}
\def\s{\sigma}
\def\div{\mathop{\rm div}}
\def\sing{{\rm Sing}\,}
\def\singg{{\rm Sing}_\ty\,}

\def\A{{\cal A}}
\def\F{{\cal F}}
\def\H{{\cal H}}
\def\W{{\bf W}}
\def\M{{\cal M}}
\def\N{{\cal N}}
\def\S{{\cal S}}

\def\eR{{\bf R}}
\def\eN{{\bf N}}
\def\Ze{{\bf Z}}
\def\Qe{{\bf Q}}
\def\Ce{{\bf C}}

\def\ty{\infty}
\def\e{\varepsilon}
\def\f{\varphi}
\def\:{{\penalty10000\hbox{\kern1mm\rm:\kern1mm}\penalty10000}}
\def\ov#1{\overline{#1}}
\def\D{\Delta}
\def\O{\Omega}
\def\pa{\partial}

\def\st{\subset}
\def\stq{\subseteq}
\def\pd#1#2{\frac{\pa#1}{\pa#2}}
\def\sgn{{\rm sign}\,}
\def\sp#1#2{\langle#1,#2\rangle}

\newcount\br@j
\br@j=0
\def\q{\quad}
\def\gg #1#2{\hat G_{#1}#2(x)}
\def\inty{\int_0^{\ty}}
\def\remark{\smallskip\advance\br@j by1 \noindent{\csc Remark
\the\br@j.}\kern3mm\relax}
\def\example{\smallskip\advance\br@j by1 \noindent{\csc Example
\the\br@j.}\kern3mm\relax}
\def\od#1#2{\frac{d#1}{d#2}}

\def\bg{\begin}
\def\eq{equation}
\def\bgeq{\bg{\eq}}
\def\endeq{\end{\eq}}
\def\bgeqnn{\bg{eqnarray*}}
\def\endeqnn{\end{eqnarray*}}
\def\bgeqn{\bg{eqnarray}}
\def\endeqn{\end{eqnarray}}

\def\bgeqq#1#2{\bgeqn\label{#1} #2\left\{\begin{array}{ll}}
\def\endeqq{\end{array}\right.\endeqn}

\def\abstract{\bgroup\leftskip=2\parindent\rightskip=2\parindent
        \noindent{\bf Abstract.\enspace}}
\def\endabstract{\par\egroup}

\def\udesno#1{\unskip\nobreak\hfil\penalty50\hskip1em\hbox{}
             \nobreak\hfil{#1\unskip\ignorespaces}
                 \parfillskip=\z@ \finalhyphendemerits=\z@\par
                 \parfillskip=0pt plus 1fil}
\catcode`\@=11

\def\cal{\mathcal}
\def\eR{\mathbb{R}}
\def\eN{\mathbb{N}}
\def\Ze{\mathbb{Z}}
\def\Qu{\mathbb{Q}}
\def\Ce{\mathbb{C}}

\def\sideremark#1{\ifvmode\leavevmode\fi\vadjust{\vbox to0pt{\vss 
    \hbox to 0pt{\hskip\hsize\hskip1em           
 \vbox{\hsize2cm\tiny\raggedright\pretolerance10000
 \noindent #1\hfill}\hss}\vbox to8pt{\vfil}\vss}}}%
\newcommand{\edz}[1]{\sideremark{#1}}

\section{Introduction} \label{section-introduction}

A fractal-dimensional analysis of a planar weak focus $(-y+\dots)\frac{\partial}{\partial x}+(x+\dots)\frac{\partial}{\partial y}$ has been done in \cite{zuzu,belg} using a box dimension approach. The box dimension of trajectories spiralling around a weak focus has been computed. Furthermore, an explicit relation between the box dimension and the leading power in the asymptotic expansion of the Poincar\' e map of the weak focus has been obtained (for more details see \cite{zuzu,belg}). The box dimension of spiral trajectories changes from trivial to nontrivial for parameter values at which some bifurcations occur (Hopf-Takens bifurcations \cite{zuzu}, Bogdanov-Takens bifurcations \cite{Lana2,DHVV}, discrete saddle-node and period doubling bifurcations \cite{elzuzu,Lana}, etc.) Our paper is a natural continuation of \cite{zuzu}. We deal with a class of planar singular points of focus type (we assume that the linear part has both eigenvalues equal to zero) and compute their box dimension.
\smallskip

We focus on a particular deformation  (to be specified later) of $X=-ny^{2n-1}\frac{\partial}{\partial x}+mx^{2m-1}\frac{\partial}{\partial y}$, with $m,n\in\mathbb{N}\setminus\{0\}$ and $m+n>2$, or $Y=(y-x^{2n})\frac{\partial}{\partial x}$. The vector field $X$ has a center at $(x,y)=(0,0)$, with $H(x,y)=x^{2m}+y^{2n}$ as a first integral, while $Y$ has the curve of singularities $C=\{y=x^{2n}\}$ and horizontal regular orbits. We call $X$, with $m,n>1$, the degenerate case and $Y$ or $X$ with $m=1$ or $n=1$ the nilpotent case.
The fractal-dimensional analysis of such a deformation is far more difficult than  analysis of weak foci due to absence of ``regularity" of spiral trajectories. For example, the radius of the spiral trajectories near the weak focus is a decreasing function (see Theorem 5 of \cite{zuzu}). This is not true if we deal with the degenerate or nilpotent case. Moreover, the Poincar\'{e} map near the focus of the perturbation of $X$ with $m\ne n$ has two different asymptotics, one along the $x$-axis and the other one along the $y$-axis. For some other examples of ``iregular" spiral trajectories and the calculation of their box dimensions see e.g. \cite{burrell,wavy,Bessel,mil,oscil}.

\smallskip

In this paper we will often speak about a trajectory ``near the origin $(x,y)=(0,0)$". This stands for its part spiralling around the origin and
contained in an open disk centered at the origin.

\smallskip

When $m=n$ and $n$ is odd in $X$, then we prove the following result (see Theorem \ref{theorem-glavni} in Section \ref{section-main}):
\begin{itemize}
    \item \textit{Let $k\in\mathbb{N}\setminus\{0\}$. Then any trajectory of the vector field \begin{equation}\label{perturb-1}X_n:=X\pm n\left(x^n y^{n-1} (x^{2n}+y^{2n})^k\frac{\partial}{\partial x}+x^{n-1} y^n (x^{2n}+y^{2n})^k\frac{\partial}{\partial y}\right)\end{equation}
  near the origin $(x,y)=(0,0)$ is Minkowski nondegenerate and its box dimension is equal to $2-\frac{2}{1+2kn}$. If $k=0$, any trajectory of (\ref{perturb-1}) near the origin is Minkowski measurable and its box dimension is equal to $1$. }
\end{itemize}
This has been proved in \cite{zuzu} for $X_1$ (see also Section \ref{section-definitions}). When the sign in $X_n$ is negative (resp. positive), we deal with a stable (resp. unstable) focus at the origin. When $n$ is even, then the origin is a center (see Remark \ref{remark-period} in Section \ref{section-degenerate-nn}). In the proof of Theorem \ref{theorem-glavni}, we construct a bi-Lipschitz map between spiral trajectories of (\ref{perturb-1}) and $\alpha$-power spirals \cite{tricot} with a well-known box dimension. We use the fact that bi-Lipschitz maps preserve the box dimension (see Section \ref{section-definitions}). Such (degenerate) focus of type $(n,n)$ has the same  asymptotic of the Poincar\' e map in each direction (see \cite{medvedeva}).
\smallskip

Based on \cite{burrell} (Remark \ref{remark-conj-b} in Section \ref{section-gen-PC}) in the general case ($m$ and $n$ may be different) we propose the following conjecture:
\begin{itemize}
    \item \textit{Let $k\in\mathbb{N}\setminus\{0\}$, $m\ge n$ and let $m,n$ be odd. Then any trajectory of  \begin{equation}\label{perturb-2}X_{m,n}:=X\pm\left(nx^m y^{n-1} (x^{2m}+y^{2n})^k\frac{\partial}{\partial x}+m x^{m-1} y^n (x^{2m}+y^{2n})^k\frac{\partial}{\partial y}\right)\end{equation}
  near the origin $(x,y)=(0,0)$ is Minkowski nondegenerate and its box dimension is equal to $2-\frac{1+\frac{n}{m}}{1+2nk}$. }
\end{itemize}
Note that $X_{n,n}=X_n$. When $m$ is even (resp. $n$ is even), then system $X_{m,n}$ is invariant under the symmetry $(x,t)\to(-x,-t)$ (resp. $(y,t)\to(-y,-t)$) and has a center at the origin $(x,y)=(0,0)$ (see Remark \ref{remark-period-gen} in Section \ref{section-gen-PC}). In the rest of the paper we suppose that $m$ and $n$ are odd.
A spiral trajectory of (\ref{perturb-2}) cannot be (bi-Lipschitz) mapped onto some regular $\alpha$-power spiral, for $m\ne n$, due to the presence of different asymptotic expansions of the Poincar\'{e} map of (\ref{perturb-2}) depending on the direction. 
\smallskip

For the case $k=0$ in (\ref{perturb-2}) we refer to Theorem \ref{theorem-glavni-generalized} in Section \ref{section-gen-PC}.

\smallskip

We obtain $X_{m,n}$ after applying transformation $F_{m,n}:\eR^2\to\eR^2$, defined by
\bgeq\label{F}
(x,y)\mapsto (\bar x,\bar y)=((\sgn x)|x|^{1/m},(\sgn y)|y|^{1/n}),\nonumber 
\endeq
to $X_1$,
and using multiplication by $mn|\bar x|^{m-1}|\bar y|^{n-1}$. The map $F_{m,n}$ transforms spiral trajectories of $X_1$ to spiral trajectories of $X_{m,n}$. It is clear that $F_{m,n}$ is not bi-Lipschitz as $(x,y)\to (0,0)$, and the box dimension is not preserved. A similar idea is used in  e.g. \cite{ggg,gt1,gt} using the  quasi-homogeneous polar coordinates in the computation of generalized Lyapunov coefficients.

\smallskip

The main advantage of working with the model (\ref{perturb-2}) is that in polar coordinates the vector field can be reduced to a  Bernoulli or linear
differential equation which can be solved. When $m=n$ (resp. $m\ne n$), we use the standard (resp. generalized) polar coordinates. See Section \ref{section-degenerate-nn} (resp. Section \ref{section-gen-PC}).

\smallskip

If $k=0$ (resp. $k>0$) in $X_n$ or $X_{m,n}$, then trajectories near the origin are comparable with exponential spirals (resp. $\alpha$--power spirals). As it will be clear in Theorem \ref{theorem-glavni} and Theorem \ref{theorem-glavni-generalized}, comparability with an $\alpha$--power spiral is not sufficient to say something about the box dimension of the trajectories near the origin.
\smallskip

The curve of singularities $C$ of $Y$ is divided into a normally attracting part $x>0$, a normally repelling part $x<0$ and a nilpotent contact point $(x,y)=(0,0)$ between them. To study dynamics of a small $O(\epsilon)$-perturbation $Y_{\epsilon}$ of $Y$ ($\epsilon\ge 0$ is a small singular perturbation parameter), one typically uses geometric singular perturbation theory due to Fenichel \cite{Fenichel} (Fenichel describes the dynamics of $Y_\epsilon$ near normally hyperbolic parts). Near the contact point where the normal hyperbolicity is lost one uses family blow-up (see \cite{DR,KS}). The deformation $Y_\epsilon$ (often called slow-fast system), with $\epsilon>0$, may have spiral trajectories (i.e., a focus) near the origin and a natural question arises: How to compute the box dimension of the spiral and does the box dimension tell us something about the type of the contact point (the codimension, cyclicity, etc.)? Instead of computing the box dimension of the spiral at level $\epsilon>0$, it is more natural to calculate the box dimension of a  slow-fast spiral when $\epsilon\to 0$ (see Definition \ref{def-lim-focus} in Section \ref{section-singular}). Such slow-fast spiral is a union of a geometric chirp (see \cite{Lap}) and a part of the curve $C$ near the origin. The geometric chirp is defined using fast and slow subsystems of $Y_\epsilon$ and so-called entry-exit relation \cite{Benoit,DM-entryexit} (see Section \ref{section-singular}). We find the box dimension of the slow-fast spiral and establish a bijective correspondence between its box dimension and the codimension of generic or non-generic contact points (see Theorem \ref{theorem-slow-fast} in Section \ref{section-singular}). One of the reasons why it is more convenient to work with slow-fast spirals ($\epsilon=0$) instead of regular spirals at level $\epsilon>0$ is that then we don't have to use the family blow-up. We point out that a box dimension approach has already been used in the planar slow-fast setting to study limit cycles configurations Hausdorff close to so-called canard cycles (hence, not contact points). See \cite{RBox,Domagoj,CHV}.

\smallskip

Although, for the sake of readability, in this paper we develop techniques for the specific perturbations of $X$ and $Y$, we believe that similar ideas can be used in a more general framework. 
\smallskip

In Section \ref{section-definitions} we give properties of the box dimension that will be often used throughout the paper. Section \ref{section-main} is devoted to the study of the box dimension of (\ref{perturb-1}) and numerical examples. In Section \ref{section-main1} we find the box dimension of spiral trajectories of (\ref{perturb-2}) for $k=0$. The fractal analysis of planar nilpotent turning points is given in Section \ref{section-singular}. In Section \ref{section-3dim} we construct a $3$-dimensional vector field having an elliptical power spiral as a trajectory.

\section{The box dimension and its properties}\label{section-definitions}

In this section we briefly recall the notion of box dimension in $\mathbb{R}^N$ (for more details we refer the reader to  \cite{Falconer,Lap,tricot} and references therein).
For a bounded set $A\st\eR^N$ we define the \emph{$\e$-neigh\-bour\-hood} of  $A$ as
$
A_\e:=\{y\in\eR^N\:d(y,A)<\e\}
$.
By the \emph{lower $s$-dimensional  Minkowski content} of $A$, for $s\ge0$, we mean
$$
\M_*^s(A):=\liminf_{\e\to0}\frac{|A_\e|}{\e^{N-s}},
$$
and analogously for the \emph{upper $s$-dimensional Minkowski content} $\M^{*s}(A)$ (we replace $\liminf_{\e\to0}$ with $\limsup_{\e\to0}$).
If $\M^{*s}(A)=\M_*^{s}(A)$, we call the common value the \emph{$s$-dimensional Minkowski content of $A$}, and denote it by $\M^s(A)$.
The \emph{lower and upper box dimensions} of $A$ are
$$
\underline\dim_BA:=\inf\{s\ge0\:\M_*^s(A)=0\}
$$
 and analogously
$\ov\dim_BA:=\inf\{s\ge0\:\M^{*s}(A)=0\}$.
If these two values coincide, we call it simply the \emph{box dimension} of $A$, and denote it by $\dim_BA$. The upper box dimension is finitely stable, i.e. $\ov\dim_B(A_1\cup A_2)=\max\{\ov\dim_BA_1,\ov\dim_BA_2\}$, with $A_1,A_2\st\eR^N$.
If $0<\M_*^d(A)\le\M^{*d}(A)<\ty$ for some $d$, then we say
 that $A$ is \emph{Minkowski nondegenerate}. In this case obviously $d=\dim_BA$.
In the case when the lower or upper $d$-dimensional Minkowski content of $A$ is equal to $0$ or $\ty$, where $d=\dim_BA$, we say that $A$ is \emph{degenerate}.
If there exists $\M^d(A)$ for some $d$ and $\M^d(A)\in(0,\ty)$, then we say that $A$ is \emph{Minkowski measurable}. 

\smallskip

Suppose that $F:A\subset \mathbb{R}^{N}\rightarrow \mathbb{R}^{M}$ is a Lipschitz map. Then 
$$\underline\dim_{B}F(A)\leq \underline\dim_{B}A, \ \ov\dim_{B}F(A)\leq \ov\dim_{B}A.$$ 
If $F:A\subset \mathbb{R}^{N}\rightarrow \mathbb{R}^{M}$ is a bi-Lipschitz map (i.e., there exist constants $\kappa_1>0$ and $\kappa_2>0$ such that
$\kappa_1\left\|x-y\right\|\leq\left\|F(x)-F(y)\right\|\leq \kappa_2 \left\|x-y\right\|$,
for every $x,y\in A$), then $$\underline\dim_{B}A=\underline\dim_{B}F(A), \ \ov\dim_{B}A=\ov\dim_{B}F(A).$$
If $F$ is a bi-Lipschitz map and $A$ is Minkowski nondegenerate, then $F(A)$ is Minkowski nondegenerate (see Theorem 4.1 in \cite{ZuZuR3}).
\smallskip

We use the following notation in Section \ref{section-singular}. For any two sequences of positive real numbers $(a_l)_{l\in\mathbb{N}}$ and $(b_l)_{l\in\mathbb{N}}$ converging to zero we write $a_l\simeq b_l$, as $l\to\ty$,
if there exist positive constants $\tilde A<\tilde B$ such that $a_l/b_l\in[\tilde A,\tilde B]$ for all $l\in\mathbb{N}$. 

\smallskip

A spiral $r=f(\varphi)$ of focus type is said to be \textit{comparable} with the $\alpha$-\textit{power spiral} $r=\varphi^{-\alpha}$ if $f(\varphi)/|\varphi|^{-\alpha}\in[\tilde A,\tilde B]$ for some positive constants $\tilde A$ and $\tilde B$ and for all $\varphi\in [1,\infty[$ (resp. $\varphi\in ]-\infty,-1]$) if the spiral has positive (resp. negative) orientation. Similarly, we say that a spiral $r=f(\varphi)$ of focus type is comparable with the exponential spiral $r=e^{-\beta\varphi}$ if $f(\varphi)/e^{-\beta\varphi}\in[\tilde A,\tilde B]$ for positive constants $\tilde A$ and $\tilde B$, a positive (resp. negative) constant $\beta$ and for all $\varphi\in [0,\infty[$ (resp. $\varphi\in ]-\infty,0]$) if the spiral has positive (resp. negative) orientation. In Section \ref{section-degenerate-nn} (resp. Section \ref{section-gen-PC}) $(r,\varphi)$ denotes the standard (resp. generalized) polar coordinates. 
\smallskip

Let $\tilde\Gamma$ be a trajectory of the vector field \[-y\frac{\partial}{\partial x}+x\frac{\partial}{\partial y}\pm(x^2+y^2)^k\left( x\frac{\partial}{\partial x}+y\frac{\partial}{\partial y} \right), \ k\ge 0,\] near the origin $(x,y)=(0,0)$, expressed in (standard) polar coordinates as $r=f(\varphi)$. If $k\ge 1$, then $\tilde\Gamma$ is comparable with the power spiral $r=\varphi^{-1/2k}$ and $\dim_B\tilde\Gamma=\frac{4k}{2k+1}$. If $k=0$, then $\tilde\Gamma$ is comparable with the exponential spiral $r=e^{\pm\varphi}$ and hence $\dim_B\tilde\Gamma=1$.
In both cases $\tilde\Gamma$ is Minkowski measurable. See Theorem 9 of \cite{zuzu}.

\section{Box dimension of degenerate focus of type \boldmath ${(n,n)}$ and numerical examples}\label{section-main}

In Section \ref{section-degenerate-nn} we give a complete study of the box dimension of spiral trajectories of $X_n$. Section \ref{section-numerics-nn} is devoted to numerical examples. 

\subsection{The box dimension of degenerate focus}\label{section-degenerate-nn}
In this section we prove a result about the box dimension of trajectories of $X_n$ near the origin. We have  

\begin{theorem}\label{theorem-glavni} Let $n\ge 1$ be odd and let $\tilde{\Gamma}$ be a trajectory of $X_n$, given in (\ref{perturb-1}), near the origin. The following statements are true.
\begin{enumerate}
    \item If $k=0$, then the spiral $\tilde{\Gamma}$ is comparable with the exponential spiral $r=e^{\pm\frac{\varphi}{n}}$,
    $\dim_B\tilde{\Gamma}=1$, and $\tilde{\Gamma}$ is Minkowski measurable.
    \item If $k>0$, then the spiral $\tilde{\Gamma}$ is comparable with the power spiral $r=\varphi^{-1/2nk}$,
\bgeq\label{dimprvi}
\dim_B\tilde{\Gamma}=2-\frac{2}{1+2kn},
\endeq
and $\tilde{\Gamma}$ is Minkowski nondegenerate.
\end{enumerate}
\end{theorem}

\begin{proof} 
We prefer to work with the system
\bgeq\label{degprvi}
\begin{array}{ccl}
\dot x&=&-y^{2n-1}\pm x^n y^{n-1} (x^{2n}+y^{2n})^k\\
\dot y&=&\phantom{-}x^{2n-1}\pm x^{n-1} y^n (x^{2n}+y^{2n})^k,
\end{array}
\endeq
obtained by dividing $X_n$ by $n$ (the trajectories near the origin remain unchanged after division by a nowhere zero factor).
It suffices to prove Theorem \ref{theorem-glavni} for \eqref{degprvi} with the negative sign (stable focus). The case with the positive sign in \eqref{degprvi} (unstable focus) can be reduced to the negative sign by reversing the time and applying the coordinate change $(x,y)\to (y,x)$. Thus, in the rest of the proof we focus on the system
\bgeq\label{degprvi+}
\begin{array}{ccl}
\dot x&=&-y^{2n-1}- x^n y^{n-1} (x^{2n}+y^{2n})^k\\
\dot y&=&\phantom{-}x^{2n-1}-x^{n-1} y^n (x^{2n}+y^{2n})^k ,
\end{array}
\endeq
with arbitrary but fixed trajectory $\tilde{\Gamma}$ near the origin, given by the initial condition $(x_0,y_0)\neq(0,0)$. In the polar coordinates $\Theta(r,\varphi)=(r\cos\varphi,r\sin\varphi)$ system (\ref{degprvi+}) becomes
\bgeq\label{degprvipolar}
\begin{array}{ccl}
	\dot r&=&r^{2n-1}\left(\cos^{2n-1}\varphi\sin\varphi-\sin^{2n-1}\varphi\cos\varphi\right) \\
	& & - r^{2nk+2n-1}\sin^{n-1}\varphi\cos^{n-1}\varphi\left(\sin^{2n}\varphi+\cos^{2n}\varphi\right)^k \\
	\dot \varphi&=&r^{2n-2}\left(\sin^{2n}\varphi+\cos^{2n}\varphi\right) 
\end{array}
\endeq
and the initial condition corresponds to $(r_0,\varphi_0)$, where $\varphi_0>0$. The spiral $\tilde{\Gamma}$ is given in polar coordinates by $r=\tilde{r}(\varphi)$, $\varphi\in[\varphi_0,\infty[$, and $\tilde{r}(\varphi_0)=r_0$. Dividing the first by the second equation in the system (\ref{degprvipolar}) we get the equation
\bgeq\label{bernoullieqn}
r'(\varphi)+p(\varphi)r(\varphi)=q(\varphi)r(\varphi)^{2nk+1} ,
\endeq
where functions $p$ and $q$ are given by
\bgeq\label{formula-p-q}
\begin{array}{ccl}
	p(\varphi)&=&\frac{\sin^{2n-1}\varphi\cos\varphi-\cos^{2n-1}\varphi\sin\varphi}{\sin^{2n}\varphi+\cos^{2n}\varphi}, \\
	q(\varphi)&=&-\sin^{n-1}\varphi\cos^{n-1}\varphi\left(\sin^{2n}\varphi+\cos^{2n}\varphi\right)^{k-1} .\end{array}
\endeq 
When $k=0$ (resp. $k>0$), in (\ref{bernoullieqn}) we deal with a linear equation (resp. Bernoulli differential equation).
First we prove Statement 2.\\
\textit{Statement 2.} Suppose that $k>0$. The idea of the proof is to construct a bi-Lipschitz equivalence between $\tilde{\Gamma}$ and some regular $\alpha$-power spiral \cite{tricot} (Section \ref{section-definitions}), with a well known box dimension, and to use the invariance of the box dimension under bi-Lipschitz mappings. We divide the proof into three parts. In the first part we explicitly find the function $\tilde{r}(\varphi)$. In the second part we introduce a radial map which we use to define the bi-Lipschitz map, and the third part is devoted to conclusions.

\textit{(a) Finding the spiral $r=\tilde{r}(\varphi)$.} Using substitution $z=r^{-2nk}$ we transform (\ref{bernoullieqn}) to the linear equation
\bgeq\label{lineareqn}
-\frac{z'(\varphi)}{2nk}+p(\varphi)z(\varphi)=q(\varphi).
\endeq
We solve (\ref{lineareqn}) using standard techniques and get the solution
\bgeq\label{solutionoflinear}
z(\varphi)=\left(\sin^{2n}\varphi+\cos^{2n}\varphi\right)^k(I(\varphi)+C),
\endeq
with $C\in\eR$ and
\bgeq\label{integralI}
I(\varphi)= 2nk\int\limits_0^{\varphi}\frac{\sin^{n-1}\tau\cos^{n-1}\tau}{\sin^{2n}\tau+\cos^{2n}\tau}\,d\tau.
\endeq
Now, we can write the function $I$ from (\ref{integralI}) in the form
\bgeq\label{integralIidentity}
I(\varphi)=K\varphi+P(\varphi) ,
\endeq
where $K:=I(2\pi)/(2\pi)$. Because $n$ is odd, it easily follows that $K>0$. The function $P$ is given by
\bgeq\label{funkcijaP}
P(\varphi) = 2nk\int\limits_{0}^{\varphi-2l\pi}\frac{\sin^{n-1}\tau\cos^{n-1}\tau}{\sin^{2n}\tau+\cos^{2n}\tau}\,d\tau - K(\varphi - 2l\pi),
\endeq
where $l=l(\varphi)$ is the largest integer such that $2l\pi\leq\varphi$, that is, $l = \lfloor \varphi/(2\pi) \rfloor$, and we used the $2\pi$-periodicity of the subintegral function. It follows from (\ref{integralIidentity}) and (\ref{funkcijaP}) that $P$ is an analytic and $2\pi$-periodic function, thus bounded.

Now respecting substitution $z=r^{-2nk}$, (\ref{solutionoflinear}) and (\ref{integralIidentity}) we get the solution of (\ref{bernoullieqn}),
\bgeq\label{solutionofbernoulli}
r(\varphi)=\left(\sin^{2n}\varphi+\cos^{2n}\varphi\right)^{-\frac{1}{2n}}(K\varphi+P(\varphi)+C)^{-\frac{1}{2nk}}.
\endeq
Introducing $K_2:=K^{-\frac{1}{2nk}}$, $P_2(\varphi):=P(\varphi)/K$ and $C_2:=C/K$ we rewrite (\ref{solutionofbernoulli}) in the form
\bgeq\label{solutionofbernoullinormalized}
r(\varphi)=K_2 \left(\sin^{2n}\varphi+\cos^{2n}\varphi\right)^{-\frac{1}{2n}}(\varphi+P_2(\varphi)+C_2)^{-\frac{1}{2nk}}.
\endeq
Thus, $\tilde{r}(\varphi)$ from the definition of $\tilde{\Gamma}$ is given by (\ref{solutionofbernoullinormalized}) with $C_2$
uniquely defined by the initial condition $(r_0,\varphi_0)$.

\textit{(b) Constructing a bi-Lipschitz map $S$.}
We take the $\alpha$-power spiral $\hat{\Gamma}$ from \cite{tricot}, given in polar coordinates by
\bgeq\label{powerspiral}
\hat{r}(\varphi)=(\varphi+\hat{C})^{-\alpha},
\endeq
with $\alpha:={\frac{1}{2nk}}\in ]0,1[$ and $\hat{C}:=C_2$.
We define a radial map $T:]0,\infty[\times[0,2\pi[\to ]0,\infty[\times[0,2\pi[$ by 
\bgeq\label{functionTdefinition}
T(r,\varphi):=(r H(r,\varphi),\varphi) ,
\endeq
where the map
$H:]0,\infty[\times[0,2\pi[\to ]0,\infty[$ is defined by
\bgeq\label{functionHdefinition}
H(r,\varphi):=K_2 \left(\sin^{2n}\varphi+\cos^{2n}\varphi\right)^{-\frac{1}{2n}}(1+P_2(\varphi)r^{2nk})^{-\frac{1}{2nk}} .
\endeq

The functions $T$ in (\ref{functionTdefinition}) and $H$ in (\ref{functionHdefinition}) are chosen in such a way that the composition $S:=\Theta\circ T\circ\Theta^{-1}:\mathbb{R}^2\setminus(0,0)\to\mathbb{R}^2\setminus(0,0)$ maps $\hat{\Gamma}$ to $\tilde{\Gamma}$, because $P_2(\varphi)$ is a $2\pi$-periodic function.

Also, notice that function $H$ in (\ref{functionHdefinition}) is bounded and bounded away from zero for $r$ sufficiently small and partial derivatives
\begin{align}
    \frac{\partial H}{\partial r}(r,\varphi)&=-r^{2nk-1}P_2(\varphi)(1+P_2(\varphi)r^{2nk})^{-1}H(r,\varphi) ,\nonumber\\
    \frac{\partial H}{\partial \varphi}(r,\varphi)&=\left(\sin^{2n}\varphi+\cos^{2n}\varphi\right)^{-1}(\cos^{2n-1}\varphi\sin\varphi-\cos\varphi\sin^{2n-1}\varphi)H(r,\varphi)\nonumber\\
    &- \frac{1}{2nk}r^{2nk}P_2'(\varphi)(1+P_2(\varphi)r^{2nk})^{-1}H(r,\varphi) \nonumber
\end{align}
exist and are bounded for sufficiently small $r$, as from (\ref{integralI}) and (\ref{integralIidentity}) it follows that $P_2'$ is also bounded. Now, all of the conditions of Lemma \ref{lema-bilip} are satisfied, so we see that function $S$ is a bi-Lipschitz map in a sufficiently small neighborhood of the origin. 

\textit{(c) Conclusions.} As box dimension and Minkowski nondegeneracy is invariant under the bi-Lipschitz mapping (see Section \ref{section-definitions}), it follows from \cite{tricot} and Theorem 6.1 in \cite{DarkoMinkowski} that
\bgeq
\dim_B\tilde{\Gamma} = \dim_B\hat{\Gamma} = 2 - \frac{2\alpha}{\alpha + 1} = 2-\frac{2}{1+2kn} ,
\endeq
and that $\tilde{\Gamma}$ is Minkowski nondegenerate. Since $P_2$ is bounded and the term in front of $(\varphi+\cdots)^{-1/2nk}$ in (\ref{solutionofbernoullinormalized}) is bounded and bounded away from zero, it is clear that the spiral $\tilde{\Gamma}$ is comparable with the power spiral $r=\varphi^{-1/2nk}$, as $\varphi\to\infty$. This completes the proof of Statement 2.

\textit{Statement 1.} Let $k=0$ and let $n\ge 1$ be odd. Using (\ref{bernoullieqn}) the spiral $\tilde{\Gamma}$ is given by 
\begin{equation}\label{k=0-1}
    \tilde{r}(\varphi)=Ce^{-\int_0^\varphi (p(\tau)-q(\tau))d\tau}
\end{equation}
where $C=r_0e^{\int_0^{\varphi_0} (p(\tau)-q(\tau))d\tau}$ and $p$ and $q$ are given in (\ref{formula-p-q}). Since the function $p$ is odd and $2\pi$-periodic, it follows that the integral $\int_0^\varphi p(\tau)d\tau$ is even and $2\pi$-periodic (thus, bounded). On the other hand, the integral $-\int_0^\varphi q(\tau)d\tau$ is equal to the integral given in (\ref{integralI}) and can therefore be written as (\ref{integralIidentity}), for some new positive constant $K:=-\int_0^{2\pi} q(\tau)d\tau/(2\pi)$ and a new $2\pi$-periodic function $P$ (for the exact value of $K$ see the end of the proof of Statement 1). Now, the expression in (\ref{k=0-1}) changes into
\begin{equation}\label{k=0-2}
    \tilde{r}(\varphi)=Ce^{-\left(K\varphi+R(\varphi)\right)}
\end{equation}
where $K$ is given above and the function $R$ is $2\pi$-periodic and thus bounded. It follows directly from (\ref{k=0-2}) that the spiral $\tilde{\Gamma}$ is comparable with the exponential spiral $r(\varphi)=e^{-K\varphi}$ as $\varphi\to \infty$. Further, the derivative of (\ref{k=0-1}) is equal to
\begin{equation}\label{k=0-3}
    \tilde{r}'(\varphi)=(q(\varphi)-p(\varphi))\tilde{r}(\varphi).
\end{equation}
Using (\ref{k=0-2}), (\ref{k=0-3}) and the fact that $p$ and $q$ are bounded, we get
\begin{equation}\label{k=0-4}
    |\tilde{r}'(\varphi)|\le Me^{-K\varphi}\nonumber
\end{equation}
where $M$ is a positive constant. This, together with (\ref{k=0-2}), implies that the length of the spiral $\tilde{\Gamma}$ is finite. Thus, $\dim_B\tilde{\Gamma}=1$ and $\tilde{\Gamma}$ is Minkowski measurable (see \cite{Falconer} and \cite[p. 106]{tricot}). Notice that the function $\tau\mapsto \frac{1}{n}\arctan\left(\tan^n \tau\right)$ is a primitive function of $-q(\tau)$, and this implies that $K=\frac{1}{n}$. This completes the proof of Statement 1.
\end{proof}

\begin{rem}\label{remark-period}
Notice that for $n$ even and $k>0$ (resp. $k=0$) it follows from (\ref{solutionofbernoulli}) (resp. (\ref{k=0-1}))  that $\tilde{r}(\varphi)$ is $2\pi$-periodic, so the system (\ref{degprvi+}) only has closed periodic trajectories and the fixed point at the origin is of center type.
\end{rem}

In the rest of this section we prove the following lemma used in the proof of Theorem \ref{theorem-glavni}.
\begin{lemma}\label{lema-bilip}
    Let $T:]0,\infty[\times\eR\to ]0,\infty[\times\eR$ be a radial map, i.e., \[T(r,\phi)=(r H(r,\phi),\phi),\] where $H:]0,\infty[\times\eR\to ]0,\infty[$ is a differentiable function. Suppose that there exist positive real constants $r_1$, $C_1$ and $C_2$ such that 
    \begin{itemize}
    \item[A.]  $C_1\leq H(r,\phi) \leq C_2$ for $r\in]0,r_1]$;
    \item[B.] the partial derivatives $\frac{\partial H}{\partial r}(r,\phi)$ and $\frac{\partial H}{\partial \phi}(r,\phi)$ are bounded for $r\in]0,r_1]$.
    \end{itemize}
    Then the map $S:\mathbb{R}^2\setminus(0,0)\to\mathbb{R}^2\setminus(0,0)$, defined by $S:=\Theta\circ T\circ\Theta^{-1}$, where $\Theta(r,\phi)=(r\cos\phi, r\sin\phi)$, has an inverse map $S^{-1}$ and both $S$ and $S^{-1}$ are Lipschitz maps in some punctured neighborhood of the origin.
\end{lemma}

\begin{proof}
It is easy to see that $T^{-1}$ exists and $S^{-1}=\Theta\circ T^{-1}\circ\Theta^{-1}$. To finish the proof we need to see that all of the components of the Jacobian matrices $J_S$ and $J_{S^{-1}}$ are bounded. It is sufficient to prove that $S$ and $S^{-1}$ are Lipschitz maps on the intersection of some punctured neighborhood of the origin and the half-plane $\mathbb{R^+}\times\mathbb{R}$. As the rotation in the plane is isometry, it follows that $S$ and $S^{-1}$ are Lipschitz maps on the intersection of some punctured neighborhood of the origin and any half-plane given by an arbitrary rotation around the origin of the $\mathbb{R^+}\times\mathbb{R}$ half-plane. So, by rotating this half-plane around the origin, we can cover a neighborhood of the origin by a finite number of such bi-Lipschitz maps, which proves the Lemma.

We proceed by computing partial derivatives of $S=(S_1,S_2)$ in respect to variables $x$ and $y$, on a half-plane $\mathbb{R^+}\times\mathbb{R}$, where $\Theta^{-1}$ is explicitly given by
\[
\Theta^{-1}(x,y)=\left(\sqrt{x^2+y^2}, \arctan\left(\frac{y}{x}\right)\right) .
\]
After meticulous calculation we get
\begin{align}
\label{lemmaizraz1}
\frac{\partial S_1}{\partial x}(x,y) & = H(r,\phi)\left(\frac{x}{r}\cos\phi + \frac{y}{r}\sin\phi\right) - \frac{\partial H}{\partial\phi}(r,\phi)\frac{y}{r}\cos\phi + \frac{\partial H}{\partial r}(r,\phi)x\cos\phi ,\\
\frac{\partial S_1}{\partial y}(x,y) & = H(r,\phi)\left(\frac{y}{r}\cos\phi - \frac{x}{r}\sin\phi\right) + \frac{\partial H}{\partial\phi}(r,\phi)\frac{x}{r}\cos\phi + \frac{\partial H}{\partial r}(r,\phi)y\cos\phi ,\\
\frac{\partial S_2}{\partial x}(x,y) & = H(r,\phi)\left(-\frac{y}{r}\cos\phi + \frac{x}{r}\sin\phi\right) - \frac{\partial H}{\partial\phi}(r,\phi)\frac{y}{r}\sin\phi + \frac{\partial H}{\partial r}(r,\phi)x\sin\phi ,\\
\label{lemmaizraz4}
\frac{\partial S_2}{\partial y}(x,y) & = H(r,\phi)\left(\frac{x}{r}\cos\phi + \frac{y}{r}\sin\phi\right) + \frac{\partial H}{\partial\phi}(r,\phi)\frac{x}{r}\sin\phi + \frac{\partial H}{\partial r}(r,\phi)y\sin\phi ,
\end{align}
where we write $r:=\sqrt{x^2+y^2}$ and $\phi:=\arctan\left(\frac{y}{x}\right)$. As $x,|y|\leq r$, from assumptions A and B we see that terms in expressions (\ref{lemmaizraz1})--(\ref{lemmaizraz4}) are bounded on the punctured ball $\overline{B}(0;r_1)$, so it follows that all components of the Jacobian $J_S$ are bounded on the punctured ball $\overline{B}(0;r_1)$.

To prove the boundness of $J_{S^{-1}}$, we use the inverse function theorem. It suffices to see that all of the components of the matrix $[J_S]^{-1}$ are bounded. This follows by an elementary computation using well known formula for the inverse of $2\times 2$ matrix and using the fact that
$$
\det J_S = H(r,\phi)\left(H(r,\phi) + r\frac{\partial H}{\partial r}(r,\phi)\right)
$$
is bounded away from zero for sufficiently small $r$, using assumptions A and B.\end{proof}

\subsection{Numerical examples}\label{section-numerics-nn}

In this section we develop a numerical estimate of the box dimension of trajectories of $X_n$ near the origin. The exact theoretical result has already been proven in the Section \ref{section-degenerate-nn}. Here we verify that result numerically.

\subsubsection{Numerical estimate of the box dimension}

Our goal here is to design a numerical scheme for estimating the box dimension of trajectory $\tilde{\Gamma}$ of $X_n$, from Theorem \ref{theorem-glavni}. We proceed by dividing trajectory $\tilde{\Gamma}$ to $\num$ disjoint parts $\tilde{\Gamma}_j$ such that
\[
\tilde{\Gamma}_j := \tilde{\Gamma} \cap \mathcal{K}_j , \quad \forall j,\ 1\leq j\leq\num ,
\]
where $\mathcal{K}_j$ is an unbounded circular sector between angles $\frac{2(j-1)\pi}{\num}$ and $\frac{2j\pi}{\num}$.

Next, we estimate the box dimension $\mathcal{D}_j$ of set $\tilde{\Gamma}_j$ for a fixed $j$. The first step is to estimate the area of $\epsilon$-neighborhood of $\tilde{\Gamma}_j$, which we designate by $|\tilde{\Gamma}_{j,\epsilon}|$. Notice that $\tilde{\Gamma}_j$ is actually a disjoint union of countably many curves of finite length, each curve spanning between two rays:
\begin{align*}
 R_{j-1}\ \ldots\ & \varphi=\frac{2(j-1)\pi}{\num} , \\
 R_j\ \ldots\ & \varphi=\frac{2j\pi}{\num} .
\end{align*}
For the purpose of computing $|\tilde{\Gamma}_{j,\epsilon}|$, from (\ref{solutionofbernoullinormalized}) it follows that for sufficiently large $\num$ and sufficiently close to the origin, we can successfully approximate these countably many curves with parts of circular arcs between rays $R_{j-1}$ and $R_j$. We denote this set of arcs with $\mathcal{A}_j$, so it approximately holds that $|\tilde{\Gamma}_{j,\epsilon}|\approx |\mathcal{A}_{j,\epsilon}|$, for any small $\epsilon>0$, that is, the area of $\epsilon$-neighborhood of $\tilde{\Gamma}_j$ is approximately the same as the area of $\epsilon$-neighborhood of $\mathcal{A}_j$. To get the radii of these circular arcs we compute $r(\varphi)$ from (\ref{solutionofbernoullinormalized}), by taking $\varphi=\varphi_j+2k\pi$, for every $k\in\mathbb{N}_0$, where $\varphi_j:=\frac{2j\pi}{\num}$. By using $2\pi$-periodicity of trigonometric functions and function $P_2$, from (\ref{solutionofbernoullinormalized}) we get
\bgeq\label{eq:radius_by_k}
r(\varphi)=K_2 \left(\sin^{2n}\varphi_j+\cos^{2n}\varphi_j\right)^{-\frac{1}{2n}}(\varphi_j+2k\pi+P_2(\varphi_j)+C_2)^{-\frac{1}{2nk}}.
\endeq
Next, we define new constants:
\begin{align*}
    \alpha & :=\frac{1}{2nk} ,\\
    K_5 & := \frac{\varphi_j+P_2(\varphi_j)+C_2}{2\pi} ,\\
    K_6 & := K_2 \left(\sin^{2n}\varphi_j+\cos^{2n}\varphi_j\right)^{-\frac{1}{2n}}2\pi^{-\alpha} ,
\end{align*}
so we can rewrite (\ref{eq:radius_by_k}) as
\bgeq\label{eq:radius_by_k_short}
r(\varphi)=K_6 (k+K_5)^{-\alpha}.
\endeq
Now each $k\in\mathbb{N}_0$ indexes a single circular arc having the radius equal to
\bgeq\label{eq:radius_k}
r_k := r(\varphi_j+2k\pi) = K_6 (k+K_5)^{-\alpha} .
\endeq

We generalize the approach from \cite{tricot}, where the concept of nucleus and tail part of $\epsilon$-neighborhood was introduced. For a fixed $\epsilon>0$, we define critical $k_1=k_1(\epsilon)$ as the smallest integer $k$ such that $r_k - r_{k+1} \leq 2\epsilon$. This critical $k_1$ divides the tail part from the nucleus part of the $\epsilon$-neighbourhood of the set $\mathcal{A}_j$. The tail part corresponds to $\epsilon$-neighbourhood of all arcs from $\mathcal{A}_j$ having indices $k\leq k_1$. These arcs all have disjoint $\epsilon$-neighbourhoods. On the contrary, the nucleus part corresponds to $\epsilon$-neighbourhood of all arcs from $\mathcal{A}_j$ having indices $k > k_1$. These arcs have overlapping $\epsilon$-neighborhoods.

To proceed, we first solve the equation $r_{\beta} - r_{\beta+1}=2\epsilon$, allowing $\beta\geq 0$ to be a real number, and then taking $k_1:=\lceil\beta\rceil$. We rewrite the equation as
\bgeq\label{eq:func_f_diff}
f(\beta+K_5)-f((\beta+K_5)+1)=\frac{2\epsilon}{K_6} ,
\endeq
where $f(x)=x^{-\alpha}$. As function $f$ is an analytic function for all $x>0$, and we can make $K_5$ arbitrarily large by choosing the initial condition close to the origin and $\epsilon$ close to $0$, we can develop the second term on the left hand side of (\ref{eq:func_f_diff}) in Taylor series around point $x_0=\beta + K_5$.
By taking only the first two terms in this Taylor series, we get an approximate form of equation (\ref{eq:func_f_diff}),
\bgeq\label{eq:jedn_diff_approx}
\alpha(\beta + K_5)^{-\alpha-1} \approx \frac{2\epsilon}{K_6} ,
\endeq
which can easily be solved. Notice that it is not possible to find solution $\beta$ of the original equation (\ref{eq:func_f_diff}) in closed form. We get the approximate solution $\beta$ and compute
\bgeq\label{eq:jedn_diff_approx_solved}
k_1=\lceil\beta\rceil = \left\lceil\left(\frac{\alpha K_6}{2\epsilon}\right)^{\frac{1}{\alpha+1}}-K_5\right\rceil .
\endeq

Last, we estimate the area of the tail and nucleus part of $|\mathcal{A}_{j,\epsilon}|$. The sum of this  tail and nucleus is then approximately equal to $|\tilde{\Gamma}_{j,\epsilon}|$.

To get the area of the tail part we first have to compute the sum $\mathcal{S}_j := \sum\limits_{k=0}^{k_1} r_k$. As $k_1$ could be very large and thus impossible for exact summation, we approximate $\mathcal{S}_j$ using method of the interval test for convergence, where we compute the approximation error to be smaller than $K_6/(K_5)^{\alpha}$. We can take $K_5$ to be sufficiently large, as before, so the approximation error can be made arbitrarily small. For sum $\mathcal{S}_j$, we get
\bgeq\label{eq:sum_of_orc_lengths}
\mathcal{S}_j \approx K_6\int\limits_{0}^{k_1+1} \frac{dt}{(t+K_5)^{\alpha}} = K_6\frac{(k_1+1+K_5)^{1-\alpha}-K_5^{1-\alpha}}{1-\alpha} .
\endeq
The area of the tail part is equal to the sum of $\epsilon$-neighborhoods of all circular arcs in the tail part of $\mathcal{A}_j$. We approximate $\epsilon$-neighborhood of every circular arc in tail part arc using the $\epsilon$-neighborhood of a segment having the same length as the arc. As all of this $\epsilon$-neighborhoods are disjoint, we compute the total sum of this $\epsilon$-neighborhoods to be $\mathcal{S}_j(2\pi/\num)\cdot 2\epsilon + (k_1+1)\epsilon^2$. On the other hand, the area of the nucleus part is equal to $(r_{k_1})^2\pi/\num$. Finally, using (\ref{eq:sum_of_orc_lengths}) we get
\bgeq\label{eq:curve_eps_neighb_area}
|\tilde{\Gamma}_{j,\epsilon}| \approx \left(K_6\frac{(k_1+1+K_5)^{1-\alpha}-K_5^{1-\alpha}}{1-\alpha}\right)\frac{4\pi}{\num}\epsilon + (k_1+1)\epsilon^2 + (r_{k_1})^2\frac{\pi}{\num} .
\endeq

We estimate the box dimension of set $\tilde{\Gamma}_i$ using a standard formula from \cite{Falconer},
\bgeq\label{eq:box_dim_formula}
\mathcal{D}_j :=  \dim_{B}\tilde{\Gamma}_j= 2 - \lim\limits_{\epsilon\to 0}\frac{\log |\tilde{\Gamma}_{j,\epsilon}|}{\log \epsilon} .
\endeq
Precisely, we approximate $\mathcal{D}_j$ by evaluating (\ref{eq:box_dim_formula}) for some fixed $\epsilon_0>0$, that is very close to zero,
\bgeq\label{eq:box_dim_formula_approx}
\mathcal{D}_j \approx 2 - \frac{\log |\tilde{\Gamma}_{j,\epsilon_0}|}{\log \epsilon_0} .
\endeq

Using the area of the tail and the nucleus part from (\ref{eq:curve_eps_neighb_area}) and substituting to (\ref{eq:box_dim_formula_approx}), we compute
\bgeq\label{eq:box_dim_formula_approx_final_form}
\mathcal{D}_j \approx 2 - \frac{\log\left[ \left(\frac{(k_1+1+K_5)^{1-\alpha}-K_5^{1-\alpha}}{1-\alpha}\right)\frac{4\pi}{\num}\epsilon_0+ (k_1+1)\epsilon_0^2 + (r_{k_1})^2\frac{\pi}{\num} \right]}{\log \epsilon_0} .
\endeq

Finally, we use estimates $\mathcal{D}_j$ from (\ref{eq:box_dim_formula_approx}) of the box dimension of $\tilde{\Gamma}_j$ to estimate upper and lower bounds on the box dimension of $\tilde{\Gamma}$. We assume that $\dim_{B}\tilde{\Gamma}_j$ exists and is approximately equal to previously obtained $\mathcal{D}_j$. We use monotonicity of the upper and lower box dimension, and also finite stability of the upper box dimension (see \cite{Falconer}) to compute required estimates. For the estimate on the upper bound, we have
\begin{equation}
 \label{eq:dimlesssim}
 \overline\dim_{B}\tilde{\Gamma} = \max_{1\leq j\leq \num} \overline\dim_{B}\tilde{\Gamma}_j \approx \max_{1\leq j\leq \num}\mathcal{D}_j .
\end{equation}
For the lower bound we first see that
\[
\mathcal{D}_j\approx \underline\dim_{B}\tilde{\Gamma}_j \leq  \underline\dim_{B}\tilde{\Gamma}, \quad \forall j,\ 1\leq j\leq\num,
\]
from which follows that
\begin{equation}
 \label{eq:simlessdim}
 \max_{1\leq j\leq \num}\mathcal{D}_j \approx  \max_{1\leq j\leq \num} \underline\dim_{B}\tilde{\Gamma}_j \leq  \underline\dim_{B}\tilde{\Gamma} .
\end{equation}
Using estimates (\ref{eq:simlessdim}) and (\ref{eq:dimlesssim}), as $\underline\dim_{B}\tilde{\Gamma}\leq\dim_{B}\tilde{\Gamma}\leq\overline\dim_{B}\tilde{\Gamma}$, we finally see that
\begin{equation}
 \label{eq:dimsim}
 \dim_{B}\tilde{\Gamma} \approx \max_{1\leq j\leq \num}\mathcal{D}_j .
\end{equation}

\subsubsection{Implementation details and test results}

We implemented the algorithm for computing our numerical estimate of the box dimension using Wolfram Mathematica, version 12. See \url{https://github.com/FRABDYN/DegenerateSpirals} where our code is available for download.

All computations have been done symbolically, using exact fractions where appropriate, having essentially infinite numerical precision. The box dimension results were only numericalized in step (\ref{eq:box_dim_formula_approx_final_form}) when computing $\mathcal{D}_j$.

\begin{table}[htb]
\begin{center}
\begin{tabular}{|r|c|l|l|c|}
\hline
$n$ & $k$ & \multicolumn{2}{c|}{\begin{tabular}[c]{@{}c@{}}theoretical\\ dimension\end{tabular}} & \begin{tabular}[c]{@{}c@{}}numerical\\ dimension\end{tabular} \\ \hline
3   & 2   & $24/13$                                   & $1.84615$                                & $1.84593$                                                     \\ \hline
3   & 11  & $132/76$                                  & $1.97015$                                & $1.96992$                                                     \\ \hline
11  & 2   & $88/45$                                   & $1.95556$                                & $1.95534$                                                     \\ \hline
11  & 11  & $484/243$                                 & $1.99177$                                & $1.99155$                                                     \\ \hline
\end{tabular}
\end{center}
\caption{Theoretical and numerical box dimensions computed for different values of $n$ and $k$.}
\label{tablica_dims}
\end{table}

We present numerical computations for few different examples having different values of $n$ and $k$ in Table \ref{tablica_dims}. For all test cases we used the initial condition values of $r_0=1/10$ and $\varphi_0=0$, corresponding to the constant $C$ from (\ref{solutionoflinear}) equal to $100^{nk}$. By computing the integral $I$ from (\ref{integralI}) we get that constant $K=2k$. Furthermore, to speed up the calculation we approximate constant $K_5$ by $C_2/(2\pi)$ as $\varphi_j+P_2(\varphi_j)$ is at least $10^{10}$ times less than $C_2$.

For all test cases we decided to use $\epsilon=10^{-10000}$ and $L=1000$. Using this values we produced results having high numerical precision compared the the theoretical result from Theorem \ref{theorem-glavni} and computations last no more than few seconds on a modern PC computer. For even higher precision, value of $\epsilon$ could be further decreased and $L$ increased, which would increase the computation time.

\section{Box dimension of degenerate focus of type \boldmath $(m,n)$}\label{section-main1}
In Section \ref{section-gen-PC} we introduce generalized polar coordinates and show that for $k=0$ the box dimension of spiral trajectories of $X_{m,n}$ is equal to one (Theorem \ref{theorem-glavni-generalized}).

\subsection{Generalized polar coordinates}\label{section-gen-PC}

We study the box dimension of spiral trajectories of  $X_{m,n}$--given in (\ref{perturb-2})--near the origin. We write the system as
\bgeq\label{degprvi-general}
\begin{array}{ccl}
\dot x&=&-ny^{2n-1}\pm nx^m y^{n-1} (x^{2m}+y^{2n})^k\\
\dot y&=&\phantom{-}m x^{2m-1}\pm m x^{m-1} y^n (x^{2m}+y^{2n})^k.  
\end{array}
\endeq
We introduce (see \cite{gt1}) the $(n,m)$--polar coordinates $(x,y)=(r^n\text{Cs}(\varphi),r^m \text{Sn}(\varphi))$ where $\text{Cs}(\varphi)$ and $\text{Sn}(\varphi)$ are a generalization of $\cos\varphi$ and $\sin\varphi$, and satisfy
\[\dot{\text{Cs}}(\varphi)=-n \text{Sn}^{2n-1}(\varphi), \ \dot{\text{Sn}}(\varphi)=m \text{Cs}^{2m-1}(\varphi)\]
and $(\text{Cs}(0),\text{Sn}(0))=(1,0)$. Notice that $\text{Cs}^{2m}(\varphi)+\text{Sn}^{2n}(\varphi)=1$, $\text{Cs}(\varphi)$ (resp. $\text{Sn}(\varphi)$) is even (resp. odd) and both are $T$-periodic, with 
\[T=\frac{2}{mn}\frac{\Gamma(\frac{1}{2m})\Gamma(\frac{1}{2n})}{\Gamma(\frac{1}{2m}+\frac{1}{2n})},\]
where $\Gamma$ is the gamma function. In these polar coordinates system (\ref{degprvi-general}) becomes
\bgeq\label{bernoullieqn-generalized}
\frac{dr}{d\varphi}=\pm \text{Sn}^{n-1}(\varphi)\text{Cs}^{m-1}(\varphi)r^{2mnk+1},
\endeq
upon division of $\dot{r}$ by $\dot{\varphi}$.
We use the following simple lemma (see \cite{gt1}) in the proof of Theorem \ref{theorem-glavni-generalized}.
\begin{lemma}\label{lemma-pol-cor}
Let $T$ be the period of the functions $\text{Cs}(\varphi)$ and $\text{Sn}(\varphi)$. The following statements are true.
\begin{enumerate}
    \item  If either $m$ or $n$ is even, then $\int_0^T\text{Sn}^{n-1}(\varphi)\text{Cs}^{m-1}(\varphi)d\varphi=0$.
    \item If both $m$ and $n$ are odd, then $\int_0^T\text{Sn}^{n-1}(\varphi)\text{Cs}^{m-1}(\varphi)d\varphi=\frac{2\pi}{mn}$.
\end{enumerate}
\end{lemma}

\begin{theorem}\label{theorem-glavni-generalized} Let $T$ be the period of the functions $\text{Cs}(\varphi)$ and $\text{Sn}(\varphi)$. If both $m$ and $n$ are odd, then the following statements are true for a spiral trajectory $\tilde{\Gamma}$ of (\ref{degprvi-general}) near the origin.
\begin{enumerate}
    \item If $k=0$, then the spiral $\tilde{\Gamma}$ is comparable with the exponential spiral $r=e^{\pm \frac{2\pi}{T m n}\varphi}$,
    $\dim_B\tilde{\Gamma}=1$, and $\tilde{\Gamma}$ is Minkowski measurable.
    \item If $k>0$, then spiral $\tilde{\Gamma}$ is comparable with the power spiral $r=\varphi^{-1/2mnk}$.
\end{enumerate}
\end{theorem}
\begin{proof}
First we prove Statement 1. When $k=0$, then the spiral $\widetilde{\Gamma}$ is given by
\begin{equation}
    \label{k=0-generalized}
    \tilde{r}(\varphi)=Ce^{\pm \int_0^\varphi\text{Sn}^{n-1}(\tau)\text{Cs}^{m-1}(\tau)d\tau}
\end{equation}
where the constant $C>0$ is uniquely determined by the initial condition on $\tilde{\Gamma}$. We used (\ref{bernoullieqn-generalized}). Now we proceed exactly as in the proof of Theorem \ref{theorem-glavni}. We have 
\begin{equation}
    \label{decomposition-generalized}
    \int_0^\varphi\text{Sn}^{n-1}(\tau)\text{Cs}^{m-1}(\tau)d\tau=\left(\frac{1}{T}\int_0^T\text{Sn}^{n-1}(\tau)\text{Cs}^{m-1}(\tau)d\tau\right)\varphi+P(\varphi),
\end{equation}
with 
\[ P(\varphi)=\int_0^{\varphi-l T}\text{Sn}^{n-1}(\tau)\text{Cs}^{m-1}(\tau)d\tau-\frac{1}{T}\int_0^T\text{Sn}^{n-1}(\tau)\text{Cs}^{m-1}(\tau)d\tau\left(\varphi-l T\right), \]
where $l$ is the largest integer such that $lT\le\varphi$, that is, $l = \lfloor \varphi/T \rfloor$ (we use that the integrand function is $T$-periodic). The function $P$ is bounded and $T$-periodic. Now, using this, Lemma \ref{lemma-pol-cor}.2, (\ref{k=0-generalized}) and (\ref{decomposition-generalized}), we have that $\tilde{\Gamma}$ is comparable with the exponential spiral $r=e^{\pm \frac{2\pi}{T m n}\varphi}$. Since $\text{Cs}(\varphi)$ and $\text{Sn}(\varphi)$ are bounded, the length of $\tilde{\Gamma}$ is finite. This completes the proof of Statement 1.
\smallskip

To prove Statement 2, it suffices to notice that the spiral $\tilde{\Gamma}$ is given by
\begin{equation}\label{conjecture-last-equ}
    \tilde{r}(\varphi)=\left(\mp2mnk\int_0^\varphi\text{Sn}^{n-1}(\tau)\text{Cs}^{m-1}(\tau)d\tau+C\right)^{-\frac{1}{2mnk}}
\end{equation}
and to use (\ref{decomposition-generalized}).
\end{proof}

\begin{rem}\label{remark-period-gen}
If either $m$ or $n$ is even, then the first statement of Lemma \ref{lemma-pol-cor} implies that system (\ref{degprvi-general}) has a center at the origin.
\end{rem}

\begin{rem}
When $n=m$, then we deal with the $(n,n)$--polar coordinates $(x,y)=(r^n\text{Cs}(\varphi),r^n \text{Sn}(\varphi))$. Notice that in Section \ref{section-main}, instead of these generalized polar coordinates, we worked with the standard polar coordinates in which the $\alpha$--power spirals \cite{tricot} are expressed. This was important in the proof of Theorem \ref{theorem-glavni} when $k>0$. 
\end{rem}

\begin{rem}\label{remark-conj-b}
In \cite{burrell} it is proved that the box dimension of the planar elliptical spiral $(x(\varphi),y(\varphi))=(\varphi^{- p_0} \cos \varphi, \varphi^{-q_0} \sin \varphi)$, $1<\varphi<\infty$, with $0<p_0\le q_0$ and $p_0<1$, is equal to $2-\frac{p_0+q_0}{1+q_0}$. If we replace $\text{Cs}(\varphi)$ and $\text{Sn}(\varphi)$ with $\cos\varphi$ and $\sin\varphi$, in the definition of $(n,m)$--polar coordinates, and (\ref{conjecture-last-equ}) with $\tilde{r}(\varphi)=\varphi^{-\frac{1}{2mnk}}$, then we obtain a natural candidate for the box dimension of the spiral $\tilde{\Gamma}$ when $k>0$ and $m\ge n$, which is equal to $2-\frac{1+\frac{n}{m}}{1+2nk}$. Note that $p_0=\frac{1}{2mk}$ and $q_0=\frac{1}{2nk}$.
\end{rem}

\section{The box dimension of slow-fast spirals near nilpotent contact points}\label{section-singular}
We consider a $C^\infty$-smooth family of Li\'{e}nard slow-fast systems
\bgeq\label{slow-fast}
\begin{array}{ccl}
\dot x&=&y-x^{2n}\\
\dot y&=&\epsilon (a+F(x,\rho))+O(\epsilon^2)
\end{array}
\endeq
where $\epsilon\ge 0$ is a (small) singular parameter, $a\sim 0\in\mathbb{R}$, $\rho\sim 0\in\mathbb{R}^m$, $n\ge 1$ is an integer, $F(x,\rho)$ and $O(\epsilon^2)$ are $C^\infty$-functions and $F(x,\rho)=-x^{2n-1}+O(x^{2n})$. We denote system \eqref{slow-fast} by $X_{\epsilon,a,\rho}$. When $\epsilon=0$, system $X_{\epsilon,a,\rho}$ has a curve of singularities given by $C=\{y-x^{2n}=0\}$. All the singularities are normally hyperbolic (i.e. precisely one eigenvalue of the linear part of $X_{0,a,\rho}$ at $p\in C$ is zero). An exception is the point $p=(0,0)\in C$ which is a nilpotent singularity (i.e., the normal hyperbolicity at the origin is lost). When $n=1$ (resp. $n>1$), we call the origin in $X_{\epsilon,a,\rho}$ a generic (resp. non-generic) contact point (see e.g. \cite{HuRo}). We focus on the fractal analysis of the slow-fast spirals near the contact point (see Definition \ref{def-lim-focus}) in both generic and non-generic case. We use the notion of box dimension in two dimensional ambient space and geometric chirps.

\begin{figure}[htb]
	\begin{center}
		\includegraphics[width=4cm,height=4cm]{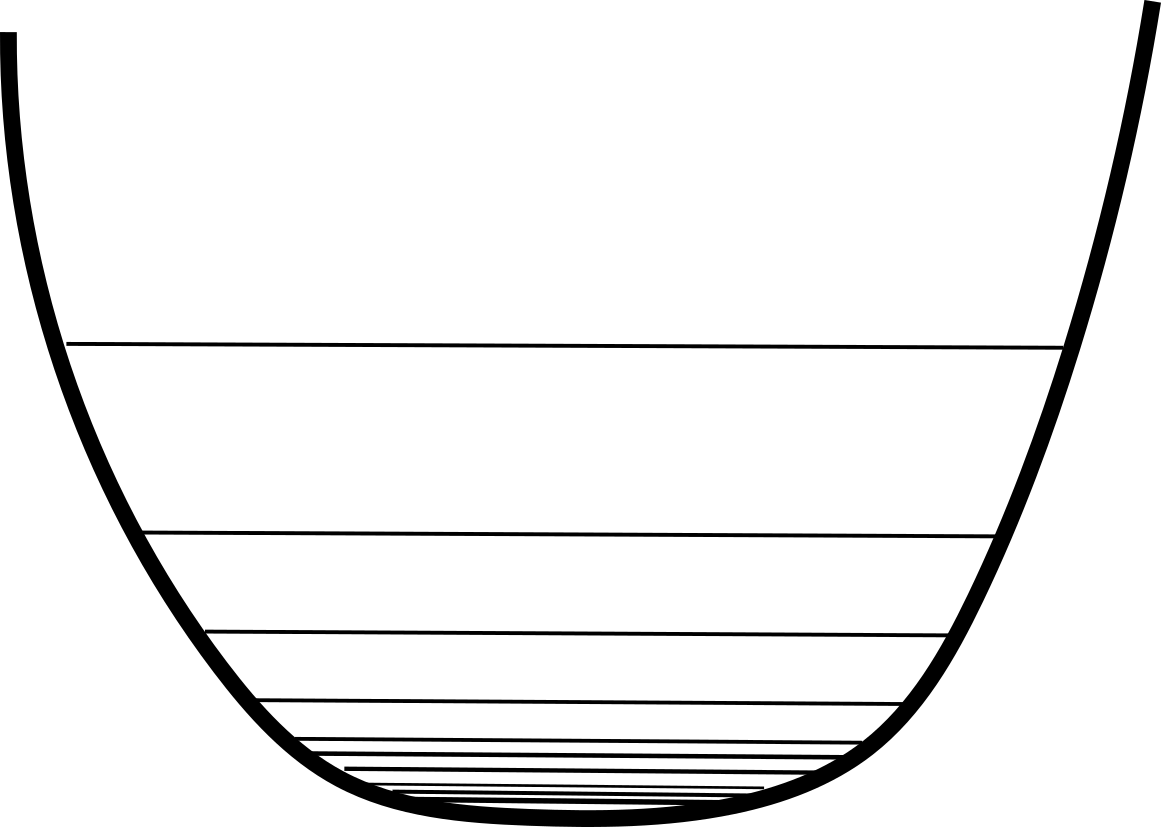}
		{\footnotesize
        \put(-62,71){$y_0$}  \put(-62,45){$y_1$} \put(-62,30){$y_2$}
}
         \end{center}
	\caption{The geometric $(\delta_1,\delta_2)$-chirp, with $\delta_1:=\frac{1}{2n}$ and $\delta_2:=\frac{2(k-n)+1}{2n}$, defined near the contact point of $X_{\epsilon,a,\rho}$ with codimension $k$ ($k\ge n$).}
	\label{figure:chirp}
\end{figure}
We denote by $\Sigma$ a section inside $\{x=0\}$, parametrized by $y\ge 0$, $y\sim 0$ ($y=0$ corresponds to the origin $(x,y)=(0,0)$).  We define the slow divergence integral along the attracting part $\{x>0 ,x\sim 0\}$ (resp. the repelling part $\{x<0,x\sim 0\}$) of $C$:
\begin{align*}
J_-(y,\rho):=\int_{\omega(y)}^0&\frac{-(2nx^{2n-1})^2dx}{F(x,\rho)}<0 \nonumber\\
& \left(\text{resp. }  J_+(y,\rho):=\int_{\alpha(y)}^0\frac{-(2nx^{2n-1})^2dx}{F(x,\rho)}<0 \right)
\end{align*}
where $(y,\rho)\sim (0,0)$, $y>0$ and $\omega(y)=y^{\frac{1}{2n}}>0$ (resp. $\alpha(y)=-y^{\frac{1}{2n}}<0$) is the $\omega$-limit (resp. $\alpha$-limit) of the fast horizontal orbit of $X_{0,a,\rho}$ through $y\in\Sigma$. The divergence of $X_{0,a,\rho}$ is given by $-2nx^{2n-1}$ and the slow dynamics along $C$ is $\frac{dx}{d\tau}=\frac{F(x,\rho)}{2nx^{2n-1}}$ where $\tau$ is the slow time. Note that $J_{\pm}$ are well-defined (i.e. finite) because the leading term of $F$ is $x^{2n-1}$. 
\smallskip

We assume that $(J_{-}-J_{+})(y,0)\ne 0$ for all $y\sim 0$ and $y>0$. (In the limit $y=0$, we have $(J_{-}-J_{+})(0,0)=0$ because $\alpha(0)=\omega(0)=0$.) If $(J_{-}-J_{+})(y,0)<0$ for $y\sim 0$ and $y>0$ (resp. $(J_{-}-J_{+})(y,0)>0$ for $y\sim 0$ and $y>0$), then the orbit $\mathcal{O}=\{y_0,y_1,y_2,\dots\}$, defined recursively by 
\begin{equation}\label{recursion}
 J_{-}(y_{l+1},0)=J_{+}(y_l,0) \ \left(\text{resp. } J_{-}(y_{l},0)=J_{+}(y_{l+1},0)\right), \ l\ge 0,
\end{equation}
with $y_0>0$ small and fixed, is decreasing and converges to zero. This is a simple consequence of the fact that the integrand in $J_\pm$ changes sign as $x$ varies through $x=0$.
\smallskip

Using the orbit $\mathcal{O}$ we define a geometric chirp near the contact point of $X_{\epsilon,a,\rho}$ (at level $(\epsilon,a,\rho)=(0,0,0)$):

\begin{equation}
    \label{chirp}
    \mathcal{U}=\bigcup_{y_l\in \mathcal{O}}U_l\subset \mathbb{R}^2, \quad U_l=]\alpha(y_l),\omega(y_l)[\times \{y_l\}.
\end{equation}
The geometric chirp $\mathcal{U}$ is the union of horizontal open intervals $]\alpha(y_l),\omega(y_l)[$ at level $y=y_l$ (see Figure \ref{figure:chirp}). The type $(\delta_1,\delta_2)$ of the geometric chirp $\mathcal{U}$ is given in Theorem \ref{theorem-slow-fast}.

\begin{definition}
\label{def-lim-focus}
Let $(J_{-}-J_{+})(y,0)<0$ (resp. $>0$), for $y\sim 0$ and $y>0$, and let $\mathcal{O}=\{y_0,y_1,y_2,\dots\}$ be the orbit with the initial point $y_0>0$ defined in (\ref{recursion}). The unstable (resp. stable) slow-fast spiral of the contact point of $X_{\epsilon,a,\rho}$, for $(\epsilon,a,\rho)=(0,0,0)$, is the union of the geometric chirp $\mathcal{U}$, defined in (\ref{chirp}), and the part of the curve of singularities $C$ between $\alpha(y_0)$ and $\omega(y_1)$ (resp. $\alpha(y_1)$ and $\omega(y_0)$). See Figure \ref{figure:SFspiral}.
\end{definition}
\begin{rem}\label{l-focus}
Following Definition \ref{def-lim-focus}, the stable slow-fast spiral consists of the intervals $U_l$, pointing from the left to the right, and a part of $C$. We follow $U_0$ until we hit $C$ in $x=\omega(y_0)$ (entry). Then we follow the curve $C$ from $x=\omega(y_0)$ to $x=\alpha(y_1)$ (exit), then $U_1$, the curve $C$ from $x=\omega(y_1)$ (entry) to $x=\alpha(y_2)$ (exit), etc. This way we ``spiral" around the origin $(x,y)=(0,0)$ (and approach the origin). We call this ``spiral" the slow-fast spiral because it contains fast and slow intervals of $X_{\epsilon,a,\rho}$ in the limit $\epsilon\to 0$ (thus, the ``spiral" is not regular). The unstable slow-fast spiral can be explained in similar fashion. 
\end{rem}
\begin{figure}[htb]
	\begin{center}
		\includegraphics[width=9cm,height=3.4cm]{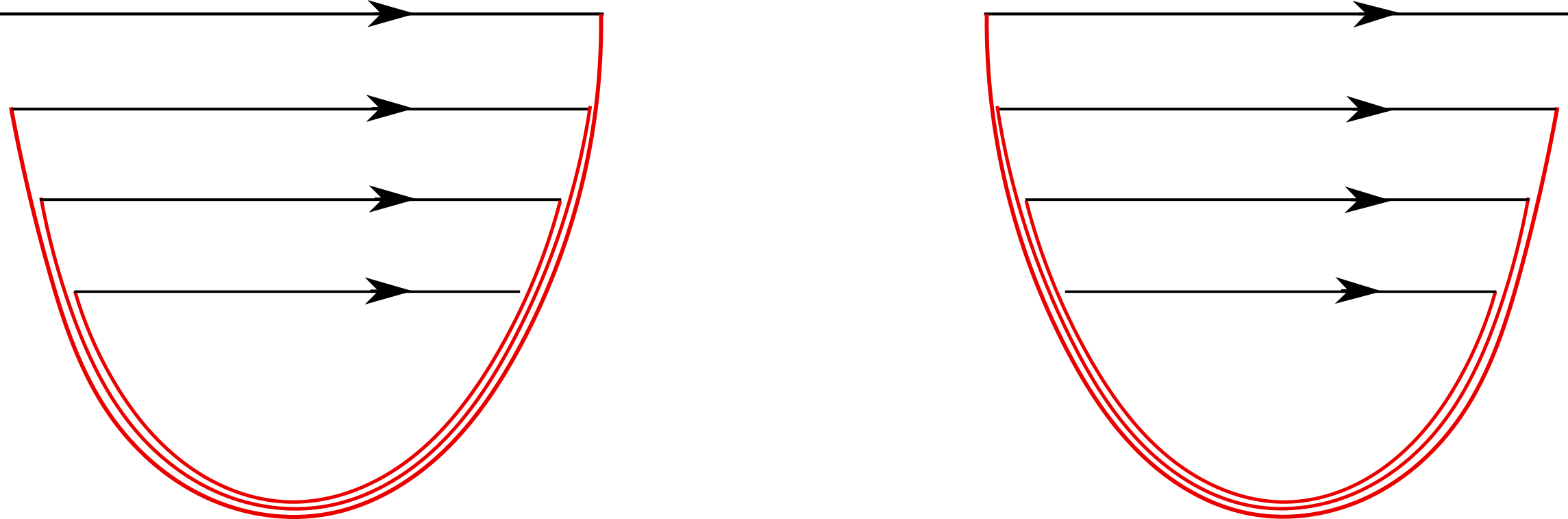}
		{\footnotesize
        \put(-57,97){$y_0$}  \put(-57,80){$y_1$} \put(-57,63){$y_2$}
        \put(-57,46){$y_3$}
        \put(-214,97){$y_0$} \put(-155,91){$\omega(y_0)$}  \put(-214,80){$y_1$}
        \put(-156,74){$\omega(y_1)$}
        \put(-214,63){$y_2$}
        \put(-159,56){$\omega(y_2)$}
        \put(-214,46){$y_3$}
        \put(-277,74){$\alpha(y_1)$}
        \put(-274,56){$\alpha(y_2)$}
        \put(1,74){$\omega(y_1)$}
        \put(-1,56){$\omega(y_2)$}
        \put(-118,91){$\alpha(y_0)$}
        \put(-118,74){$\alpha(y_1)$}
        \put(-116,56){$\alpha(y_2)$}
        \put(-52,-12){$(b)$}
        \put(-214,-12){$(a)$}
}
         \end{center}
	\caption{Slow-fast spirals with $x=\omega(y_l)$ as entry and $x=\alpha(y_l)$ as exit for all $l\in\mathbb{N}$. The slow segments are contained in the critical curve $C$ and the slow dynamics along $C$ points from the right to the left. (a) The stable slow-fast spiral. (b) The unstable slow-fast spiral.}
	\label{figure:SFspiral}
\end{figure}
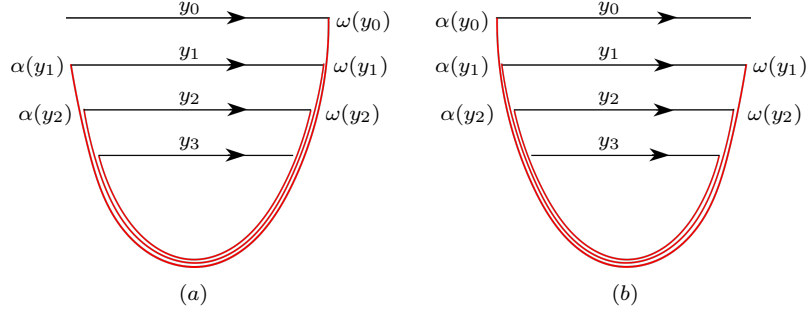

\begin{rem}
The upper box dimension of the stable or unstable slow-fast spiral from Definition \ref{def-lim-focus} is equal to the upper box dimension of the geometric chirp $\mathcal{U}$ defined in (\ref{chirp}) because the upper box dimension is finitely stable (see Section \ref{section-definitions}), $\dim_BC=1$ and $\overline{\dim}_B\mathcal{U}\ge 1$. In the rest of this section we therefore focus on the computation of the upper box dimension of $\mathcal{U}$. 
\end{rem}

\smallskip

Let denote by $f_i(\rho)$, $i\ge 2n$, the coefficients of the Taylor expansion of $F$ at $x=0$, i.e. $j^\infty F(x,\rho)=-x^{2n-1}+\sum_{i=2n}^\infty f_i(\rho)x^i$. If there exists a nonzero even coefficient $f_{2k}(0)$, we say that $X_{\epsilon,a,\rho}$ has a finite codimension (the smallest $k\ge n$ with this property is the codimension of the contact point in $X_{\epsilon,a,\rho}$). In the generic case ($n=1$), a similar definition of the codimension can be found in \cite{DRbirth}.

\begin{theorem}\label{theorem-slow-fast}
Let $y_0>0$ be small and fixed and let $\mathcal{O}=\{y_0,y_1,y_2,\dots\}$ be the orbit defined by \eqref{recursion} tending monotonically to $y=0$. Suppose that the codimension of the contact point in $X_{\epsilon,a,\rho}$ is finite and equal to $k$. Then we have $y_l\simeq l^{-\frac{2n}{2(k-n)+1}}$ as $l\to\infty$, $y_l-y_{l+1}\simeq l^{-\frac{2k+1}{2(k-n)+1}}$ as $l\to\infty$ and $\dim_B\mathcal{O}=\frac{2(k-n)+1}{2k+1}\in ]0,1[$. Moreover, $\mathcal{U}$ is the geometric $(\delta_1,\delta_2)$-chirp, with $\delta_1:=\frac{1}{2n}$ and $\delta_2:=\frac{2(k-n)+1}{2n}$, and  $\overline{\dim}_B\mathcal{U}=\frac{4k-2n+1}{2k+1}\in [1,2[$. The box dimensions are independent of the initial point $y_0$.
\end{theorem}
\begin{proof}
Suppose that the assumptions of Theorem \ref{theorem-slow-fast} are satisfied. The function $J_{-}-J_{+}$, for $\rho=0$, can be written as 
\begin{equation}
    \label{slow-fast-eq-1}
    (J_{-}-J_{+})(y,0)=\int_{-y^{\frac{1}{2n}}}^{y^{\frac{1}{2n}}}\frac{4n^2x^{2n-1}dx}{D(x)}
\end{equation}
where $y>0$, $y\sim 0$, $f_i:=f_i(0)$, $k\ge n$ is the codimension of $X_{\epsilon,a,\rho}$ ($f_{2k}\ne 0$) and $D(x)=-1+\sum_{i=n}^{k-1}f_{2i+1}x^{2(i-n+1)}+f_{2k}x^{2(k-n)+1}+O(x^{2(k-n+1)})$. The integrand function in \eqref{slow-fast-eq-1} has the following form:
\begin{equation}
    \label{slow-fast-eq-2}
    \frac{4n^2x^{2n-1}}{-1+\sum_{i=n}^{k-1}f_{2i+1}x^{2(i-n+1)}}-4n^2f_{2k}x^{2k}(1+O(x)),
\end{equation}
where the first term is an odd function. From \eqref{slow-fast-eq-1} and \eqref{slow-fast-eq-2} follows now that 
\begin{equation}
    \label{slow-fast-eq-3}
    (J_{-}-J_{+})(y,0)=-\frac{8n^2}{2k+1}f_{2k}y^{\frac{2k+1}{2n}}(1+o(1))
\end{equation}
where $o(1)$ tends to zero when $y\to 0$. In the rest of the proof we assume that $(J_{-}-J_{+})(y,0)<0$ for $y\sim 0$ and $y>0$, i.e. $f_{2k}>0$ (the case where $(J_{-}-J_{+})(y,0)>0$ for $y\sim 0$ and $y>0$ ($f_{2k}<0$) can be treated in a similar way).  We have 
\begin{align}
\label{slow-fast-eq-4}
(J_{-}-J_{+})(y_l,0)=&\int_{-y_l^{\frac{1}{2n}}}^{y_l^{\frac{1}{2n}}}\frac{(2nx^{2n-1})^2dx}{F(x,0)}\nonumber\\
=& \int_{-y_l^{\frac{1}{2n}}}^{y_l^{\frac{1}{2n}}}\frac{(2nx^{2n-1})^2dx}{F(x,0)}-\int_{-y_l^{\frac{1}{2n}}}^{y_{l+1}^{\frac{1}{2n}}}\frac{(2nx^{2n-1})^2dx}{F(x,0)}\nonumber\\
=& -\int_{y_l^{\frac{1}{2n}}}^{y_{l+1}^{\frac{1}{2n}}}\frac{(2nx^{2n-1})^2dx}{F(x,0)}\nonumber\\
=& \int_{y_l}^{y_{l+1}}2n(1+o(1))du
\end{align}
where in the second step we use \eqref{recursion} ($\int_{-y_l^{\frac{1}{2n}}}^{y_{l+1}^{\frac{1}{2n}}}=0$) and in the last step we use the coordinate change $x^{2n}=u$ (the $o(1)$-term in the last integral tends to zero as $u\to 0$). Note that the integrand function in \eqref{slow-fast-eq-4} is positive and at least continuous in $u\ge 0$ and $u\sim 0$. Finally, The Mean Value Theorem for Integrals, \eqref{slow-fast-eq-3} and \eqref{slow-fast-eq-4} imply 
\begin{equation}
    \label{slow-fast-eq-5}
    y_l-y_{l+1}\simeq y_l^{\frac{2k+1}{2n}}, \ l\to\infty.
\end{equation}
Since $\gamma:=\frac{2k+1}{2n}>1$ ($k\ge n$), Theorem 1 of \cite{elzuzu} implies that 
\begin{equation}\label{slow-fast-eq-6}
    y_l\simeq l^{-\frac{2n}{2(k-n)+1}}, \  l\to \infty.
\end{equation} This together with \eqref{slow-fast-eq-5} implies 
\[y_l-y_{l+1}\simeq l^{-\frac{2k+1}{2(k-n)+1}}, \ l\to\infty.\]
Using Theorem 1 of \cite{elzuzu} once more we get 
\[\dim_B\mathcal{O}=1-\frac{1}{\gamma}=\frac{2(k-n)+1}{2k+1}\in ]0,1[.\]
The results are clearly independent of the chosen $y_0>0$ (see \cite{elzuzu}).
\smallskip

It remains to find the upper box dimension of the geometric chirp $\mathcal{U}$. We write $\mathcal{U}=\mathcal{U}_1\cup \mathcal{U}_2$ where
\[\mathcal{U}_1=\bigcup_{y_l\in \mathcal{O}} ]-y_l^{\frac{1}{2n}},0]\times \{y_l\}, \ \ \mathcal{U}_2=\bigcup_{y_l\in \mathcal{O}} [0,y_l^{\frac{1}{2n}}[\times \{y_l\}.\]
From Section 3.6.1 in \cite{Lap} and \eqref{slow-fast-eq-6} follows that $\mathcal{U}_1$ and $\mathcal{U}_2$ are geometric $(\delta_1,\delta_2)$-chirps where $\delta_1:=\frac{1}{2n}$ and $\delta_2:=\frac{2(k-n)+1}{2n}$. Now we have (see \cite{Lap} once more) 
\[\overline{\dim}_B{\mathcal{U}_1}=\overline{\dim}_B{\mathcal{U}_2}=\max\left\{1,2-\frac{1+\delta_1}{1+\delta_2}\right\}=\frac{4k-2n+1}{2k+1}\in [1,2[.\]
This completes the proof of Theorem \ref{theorem-slow-fast} since $\overline{\dim}_B{\mathcal{U}}=\max\left\{\overline{\dim}_B{\mathcal{U}_1},\overline{\dim}_B{\mathcal{U}_2}\right\}$.
\end{proof}
\begin{rem}
Let $n\ge 1$ be fixed. Following Theorem \ref{theorem-slow-fast}, there is a one-one correspondence between $\dim_B\mathcal{O}$ (or $\overline{\dim}_B{\mathcal{U}}$) and the codimension $k$ of the contact point of $X_{\epsilon,a,\rho}$. When $k\to\infty$, then $\dim_B\mathcal{O}\to 1$ and $\overline{\dim}_B{\mathcal{U}}\to 2$.
\end{rem}
\begin{rem}
In \cite{CHV} it has been proved that $\dim_B\mathcal{O}=\frac{2k-1}{2k+1}$ in the generic case.
\end{rem}

\section{Box dimension of $3$-dimensional spiral}\label{section-3dim}

Using results from \cite{burrell} for the planar elliptical spiral $\tilde\Gamma$
\bgeq\label{fspiral}
\begin{array}{ccl}
x(t)&=&t^{-p_0} \cos t \\
 y(t)&=& t^{-q_0} \sin t,  
\end{array}
\endeq
where $0<p_0\le q_0\le 1$, we can obtain the box dimension of trajectory of $3$-dimensional systems. We  formulate this result as an example; a generalization is possible combining the results from \cite{burrell,KVZ2015}.

\begin{examp} 

The spiral (\ref{fspiral}) is  the projection of a trajectory of the system 
\begin{eqnarray}\label{sistpq}
\dot x&=&-y-p_0xz^{q_0-p_0+1} \nonumber  \\
\dot y&=&  xz^{2(q_0-p_0)}-q_0yz^{q_0-p_0+1} \\
\dot z&=& -z^{2+q_0-p_0}, \nonumber
\end{eqnarray}
to the $(x,y)$-plane.  
The system (\ref{sistpq}) is obtained by computing the derivative of  (\ref{fspiral}), using $z=t^{-1}$, and multiplying it by $z^{q_0-p_0}>0$. The parametrization of the initial curve has been changed, but not the curve itself, so the box dimension of (\ref{fspiral}) has been preserved, and equal to (see \cite{burrell}) 
$$
\dim_B\tilde\Gamma =2-\frac{p_0+q_0}{1+q_0}.
$$

Using the parametrization 
\begin{eqnarray}\label{param}
x(t)&=&t^{-p_0} \cos t \nonumber  \\
 y(t)&=& t^{-q_0} \sin t \\  \nonumber
 z(t)&=& t^{-1},  \nonumber
\end{eqnarray}
obtained before the time rescaling,
we can compute the invariant surface $ \frac{x^2}{z^{2p_0}} + \frac{y^2}{z^{2q_0}}=1 $ containing the trajectories of (\ref{sistpq}). 
Derivatives  $\frac{\partial z}{\partial x}$ and  $\frac{\partial z}{\partial y}$ are bounded, so the map $z(x,y)$ is Lipshitz.
Using \cite{ZuZuR3} we can conclude that the $3$-dimensional trajectory of the system (\ref{sistpq}) has the same box dimension as the projection curve $\tilde\Gamma$.
\smallskip

The projection $\tilde\Gamma_{xz}$ of the trajectory of (\ref{sistpq}) to the $(x,z)$-plane is a curve called chirp  $x(z)=z^{p_0}\cos 1/z$,  for $z>0$ small.
We use (see \cite{tricot}) 
\begin{equation}\label{standard_chirp}\nonumber
X_{\a,\b}(\tau)=\tau^{\a}\sin(\tau^{-\beta})\nonumber.
\end{equation}
For $0<\a\le\b$ we have
\begin{equation}\label{dim}\nonumber
\dim_BX_{\a,\b}=2-\frac{\a+1}{\b+1}\nonumber,
\end{equation}
and obtain
$$
\dim_B\tilde\Gamma_{xz} =\frac32-\frac{p_0}{2}.
$$ 
Analogously, the projection $\tilde\Gamma_{yz}$ of the trajectory of (\ref{sistpq}) to $(y,z)$-plane is   $y(z)=z^{q_0}\sin 1/z$,  with the box dimension
$$
\dim_B\tilde\Gamma_{yz} =\frac32-\frac{q_0}{2}.
$$ 

\end{examp}

\section*{Acknowledgments}

This research was supported by: Croatian Science Foundation (HRZZ) grant PZS-2019-02-3055 from “Research Cooperability” program funded by the European Social Fund.

\bibliographystyle{abbrv}
\bibliography{bibtexMobius}

\def\cprime{$'$} \def\cprime{$'$}
\begin{thebibliography}{10}

\bibitem{Benoit}
E.~Benoit.
\newblock \'{E}quations diff\'erentielles: relation entr\'ee--sortie.
\newblock {\em C. R. Acad. Sci. Paris S\'er. I Math.}, 293(5):293--296, 1981.

\bibitem{burrell}
S.~Burrell, K.~Falconer, and J.~Fraser.
\newblock The fractal structure of elliptical polynomial spirals.
\newblock {\em preprint (2020)}.

\bibitem{DM-entryexit}
P.~De~Maesschalck and F.~Dumortier.
\newblock Time analysis and entry-exit relation near planar turning points.
\newblock {\em J. Differential Equations}, 215(2):225--267, 2005.

\bibitem{DR}
F.~Dumortier and R.~Roussarie.
\newblock Canard cycles and center manifolds.
\newblock {\em Mem. Amer. Math. Soc.}, 121(577):x+100, 1996.
\newblock With an appendix by Li Chengzhi.

\bibitem{DRbirth}
F.~Dumortier and R.~Roussarie.
\newblock Birth of canard cycles.
\newblock {\em Discrete Contin. Dyn. Syst. Ser. S}, 2(4):723--781, 2009.

\bibitem{elzuzu}
N.~Elezovi{\'c}, V.~{\v{Z}}upanovi{\'c}, and D.~{\v{Z}}ubrini{\'c}.
\newblock Box dimension of trajectories of some discrete dynamical systems.
\newblock {\em Chaos Solitons Fractals}, 34(2):244--252, 2007.

\bibitem{Falconer}
K.~Falconer.
\newblock {\em Fractal Geometry}.
\newblock John Wiley and Sons, Ltd., Chichester, 1990.
\newblock Mathematical foundations and applications.

\bibitem{Fenichel}
N.~Fenichel.
\newblock Geometric singular perturbation theory for ordinary differential
  equations.
\newblock {\em J. Differential Equations}, 31(1):53--98, 1979.

\bibitem{ggg}
I.~A. Garc\'{\i}a, H.~Giacomini, and M.~Grau.
\newblock Generalized {H}opf bifurcation for planar vector fields via the
  inverse integrating factor.
\newblock {\em J. Dynam. Differential Equations}, 23(2):251--281, 2011.

\bibitem{gt1}
A.~Gasull and J.~Torregrosa.
\newblock A new algorithm for the computation of the {L}yapunov constants for
  some degenerated critical points.
\newblock In {\em Proceedings of the {T}hird {W}orld {C}ongress of {N}onlinear
  {A}nalysts, {P}art 7 ({C}atania, 2000)}, volume~47, pages 4479--4490, 2001.

\bibitem{gt}
A.~Gasull and J.~Torregrosa.
\newblock A new approach to the computation of the {L}yapunov constants.
\newblock volume~20, pages 149--177. 2001.
\newblock The geometry of differential equations and dynamical systems.

\bibitem{Lana}
L.~Horvat~Dmitrovi\'{c}.
\newblock Box dimension and bifurcations of one-dimensional discrete dynamical
  systems.
\newblock {\em Discrete Contin. Dyn. Syst.}, 32(4):1287--1307, 2012.

\bibitem{Lana2}
L.~Horvat~Dmitrovi\'{c}.
\newblock Box dimension of {N}eimark-{S}acker bifurcation.
\newblock {\em J. Difference Equ. Appl.}, 20(7):1033--1054, 2014.

\bibitem{DHVV}
L.~Horvat~Dmitrovi\'{c}, R.~Huzak, D.~Vlah, and V.~\v{Z}upanovi\'{c}.
\newblock Fractal analysis of planar nilpotent singularities and numerical
  applications.
\newblock {\em J. Differential Equations}, 293:1--22, 2021.

\bibitem{RBox}
R.~Huzak.
\newblock Box dimension and cyclicity of canard cycles.
\newblock {\em Qual. Theory Dyn. Syst.}, 17(2):475--493, 2018.

\bibitem{CHV}
R.~Huzak, V.~Crnkovi\'{c}, and D.~Vlah.
\newblock Fractal dimensions and two-dimensional slow-fast systems.
\newblock {\em J. Math. Anal. Appl.}, 501(2):Paper No. 125212, 21, 2021.

\bibitem{HuRo}
R.~Huzak and D.~Rojas.
\newblock Period function of planar turning points.
\newblock {\em Electron. J. Qual. Theory Differ. Equ.}, pages Paper No. 16, 21,
  2021.

\bibitem{Domagoj}
R.~Huzak and D.~Vlah.
\newblock Fractal analysis of canard cycles with two breaking parameters and
  applications.
\newblock {\em Commun. Pure Appl. Anal.}, 18(2):959--975, 2019.

\bibitem{wavy}
L.~Korkut, D.~Vlah, D.~\v{Z}ubrini\'{c}, and V.~\v{Z}upanovi\'{c}.
\newblock Wavy spirals and their fractal connection with chirps.
\newblock {\em Math. Commun.}, 21(2):251--271, 2016.

\bibitem{KVZ2015}
L.~Korkut, D.~Vlah, and V.~\v{Z}upanovi\'{c}.
\newblock Geometrical properties of systems with spiral trajectories in {$\Bbb
  R^3$}.
\newblock {\em Electron. J. Differential Equations}, pages No. 276, 19, 2015.

\bibitem{Bessel}
L.~Korkut, D.~Vlah, and V.~\v{Z}upanovi\'{c}.
\newblock Fractal properties of {B}essel functions.
\newblock {\em Appl. Math. Comput.}, 283:55--69, 2016.

\bibitem{KS}
M.~Krupa and P.~Szmolyan.
\newblock Relaxation oscillation and canard explosion.
\newblock {\em J. Differential Equations}, 174(2):312--368, 2001.

\bibitem{Lap}
M.~L. Lapidus, G.~Radunovi\'{c}, and D.~\v{Z}ubrini\'{c}.
\newblock {\em Fractal zeta functions and fractal drums}.
\newblock Springer Monographs in Mathematics. Springer, Cham, 2017.
\newblock Higher-dimensional theory of complex dimensions.

\bibitem{medvedeva}
N.~B. Medvedeva.
\newblock On the analytic solvability of the problem of distinguishing between
  a center and a focus.
\newblock {\em Tr. Mat. Inst. Steklova}, 254(Neline\u{\i}n. Anal. Differ.
  Uravn.):11--100, 2006.

\bibitem{mil}
S.~Mili\v{c}i\'{c}.
\newblock Box-counting dimensions of generalised fractal nests.
\newblock {\em Chaos Solitons Fractals}, 113:125--134, 2018.

\bibitem{oscil}
J.-P. Rolin, D.~Vlah, and V.~\v{Z}upanovi\'{c}.
\newblock Oscillatory integrals and fractal dimension.
\newblock {\em Bull. Sci. Math.}, 168:Paper No. 102972, 31, 2021.

\bibitem{tricot}
C.~Tricot.
\newblock {\em Curves and fractal dimension}.
\newblock Springer-Verlag, New York, 1995.
\newblock With a foreword by Michel Mend\`es France, Translated from the 1993
  French original.

\bibitem{DarkoMinkowski}
D.~\v{Z}ubrini\'{c}.
\newblock Analysis of {M}inkowski contents of fractal sets and applications.
\newblock {\em Real Anal. Exchange}, 31(2):315--354, 2005/06.

\bibitem{zuzu}
D.~\v{Z}ubrini\'{c} and V.~\v{Z}upanovi\'{c}.
\newblock Fractal analysis of spiral trajectories of some planar vector fields.
\newblock {\em Bull. Sci. Math.}, 129(6):457--485, 2005.

\bibitem{ZuZuR3}
D.~\v{Z}ubrini\'{c} and V.~\v{Z}upanovi\'{c}.
\newblock Fractal analysis of spiral trajectories of some vector fields in
  {$\Bbb R^3$}.
\newblock {\em C. R. Math. Acad. Sci. Paris}, 342(12):959--963, 2006.

\bibitem{belg}
D.~{\v{Z}}ubrini{\'c} and V.~{\v{Z}}upanovi{\'c}.
\newblock Poincar\'e map in fractal analysis of spiral trajectories of planar
  vector fields.
\newblock {\em Bull. Belg. Math. Soc. Simon Stevin}, 15(5, Dynamics in
  perturbations):947--960, 2008.

\end{thebibliography}

\end{document}